\documentclass[10pt,reqno]{amsart}
\usepackage{amssymb,mathrsfs,color}
\usepackage{pinlabel}

\usepackage{enumerate}

\usepackage{graphicx}
\usepackage{xcolor} 
\usepackage{tensor}
\usepackage{slashed}
\usepackage{cite}
\definecolor{green}{rgb}{0,0.8,0} 

\newcommand{\brk}[1]{\langle#1\rangle}
\newcommand{\set}[1]{\{#1\}}

\newcommand{\tr}{\textrm{tr}}

\newcommand{\ud}{\mathrm{d}}
\newcommand{\rd}{\partial}

\newcommand{\bb}{\Big}

\newcommand{\eps}{\epsilon}

\newcommand{\lmb}{\lambda}

\newcommand{\omg}{\omega}

\newcommand{\bfh}{{\bf h}}

\newcommand{\bbC}{\mathbb C}

\newcommand{\bbR}{\mathbb R}
\newcommand{\bbS}{\mathbb S}

\newcommand{\calE}{\mathcal E}

\newcommand{\calL}{\mathcal L}

\newcommand{\calN}{\mathcal N}

\usepackage{wrapfig}
\usepackage{tikz}
\usetikzlibrary{arrows,calc,decorations.pathreplacing}
\definecolor{light-gray1}{gray}{0.90}
\definecolor{light-gray2}{gray}{0.80}

\definecolor{deepgreen}{cmyk}{1,0,1,0.5}

\newcommand{\E}{\mathcal{E}}

\newcommand{\LL}{\mathcal{L}}
\newcommand{\NN}{\mathcal{N}}

\newcommand{\C}{\mathbb{C}}

\newcommand{\N}{\mathbb{N}}
\newcommand{\R}{\mathbb{R}}
\newcommand{\Sp}{\mathbb{S}}
\newcommand{\Z}{\mathbb{Z}}

\newcommand{\h}{\mathbf{h}}

\newcommand{\m}{\mathbf{m}}

\newcommand{\al}{\alpha}
\newcommand{\be}{\beta}

\newcommand{\om}{\omega}
\newcommand{\la}{\lambda}

\newcommand{\s}{\sigma}

\newcommand{\p}{\partial}
\newcommand{\na}{\nabla}

\newcommand{\loc}{\operatorname{loc}}

\makeatletter

\newcommand{\Rmnum}[1]{\expandafter\@slowromancap\romannumeral #1@}
\makeatother

\newcommand{\ti}{\widetilde}

\newcommand{\ang}[1]{\left\langle{#1}\right\rangle}
\newcommand{\abs}[1]{\left\lvert{#1}\right\rvert}

\newcommand{\EQ}[1]{\begin{equation}\begin{split} #1 \end{split}\end{equation}}

\newcommand{\Del}[1]{}

\numberwithin{equation}{section}

\newtheorem{thm}{Theorem}[section]
\newtheorem{cor}[thm]{Corollary}
\newtheorem{lem}[thm]{Lemma}
\newtheorem{prop}[thm]{Proposition}

\newtheorem{conj}{Conjecture}

\theoremstyle{remark}
\newtheorem{rem}{Remark}

\newcommand{\mand}{{\ \ \text{and} \ \  }}

\newcommand{\mfor}{{\ \ \text{for} \ \ }}
\newcommand{\mas}{{\ \ \text{as} \ \ }}

\newcommand{\rst}{\!\upharpoonright}	

\begin{document}

\title[A refined threshold theorem for  wave maps into surfaces]{A refined threshold theorem for  $(1+2)$-dimensional wave maps into surfaces}
\author{Andrew Lawrie}
\author{Sung-Jin Oh}

\begin{abstract}
The recently established threshold theorem for energy critical wave maps states that wave maps with energy less than that of the ground state (i.e., a minimal energy nontrivial harmonic map) are globally regular and scatter on $\bbR^{1+2}$. 
In this note we give a refinement of this theorem when the target is a closed orientable surface by taking into account an additional  invariant of the problem, namely the topological degree. We show that the sharp energy threshold for global regularity and scattering is in fact \emph{twice} the energy of the ground state for wave maps with degree zero, whereas wave maps with nonzero degree necessarily have at least the energy of the ground state.
We also give a discussion on the formulation of a refined threshold conjecture for the energy critical $SU(2)$ Yang-Mills equation on $\bbR^{1+4}$.
\end{abstract}

\thanks{The first author is an NSF Postdoctoral Fellow and support of the National Science Foundation,  DMS-1302782, is acknowledged. The second author is a Miller Research Fellow, and acknowledges support from the Miller Institute. The authors also thank Sohrab Shahshahani for helpful comments on the preliminary draft of this note.}

\maketitle

\section{Introduction}


The subject of this note is energy critical wave maps $$\Phi: (\R^{1+2}, \m) \to (\NN, \h),$$ where  $\m$ is the Minkowski metric and $(\NN, \h)$ is a Riemannian manifold with a metric $\h$. These maps arise in the physical literature as examples of nonlinear $\s$-models. Particular importance is granted to the case where the target manifold admits nontrivial finite energy static wave maps, i.e.,  harmonic maps, as these give basic examples of  topological (albeit unstable) solitons.  

Spectacular progress has been made in  recent years on the global asymptotic behavior of large energy wave maps on $\bbR^{1+2}$, culminating in the following remarkable \emph{threshold theorem} \cite{ST1, ST2, KS, Tao37}: Every wave map with energy less than that of the first nontrivial harmonic map is globally regular on $\bbR^{1+2}$ and scatters (see Theorem~\ref{t:st} for a more precise statement). In this note we refine the threshold theorem by taking into account an additional invariant of a wave map, namely its topological degree, which is available in the case when $\NN$ is a surface. A simple version reads as follows (see Theorem~\ref{t:main} for the precise formulation):
\begin{thm}[Main theorem, simple version] \label{t:main-simple}
Let $\NN$ be an orientable closed surface. Consider a smooth wave map $\Phi$ on a subset $I \times \bbR^{2}$ of $\bbR^{1+2}$ with topological degree zero. If the energy of $\Phi$ is less than \emph{twice} that of the energy of the first nontrivial harmonic map $Q : \bbR^{2} \to \NN$, then $\Phi$ is globally regular (i.e., $\Phi$ extends to $\bbR^{1+2}$ as a smooth wave map) and scatters.
\end{thm}
By the (topological) degree of a wave map, we mean the degree of its restriction $\Phi \rst_{t = t_{0}} : \bbR^{2} \to \NN$ to any constant $t$-hypersurface.  The significance of Theorem~\ref{t:main-simple} lies in the fact that if the degree of $\Phi$ is nonzero, then its energy is automatically no less than that of the first nontrivial harmonic map $Q$; see \eqref{bf} and \eqref{gs}. Hence Theorem~\ref{t:main-simple} provides an improved energy threshold precisely in the case when the general threshold theorem applies as well. The energy threshold in Theorem~\ref{t:main-simple} is moreover sharp; see Remark~\ref{r:equiv}.

To proceed to a more detailed discussion of our result, we now present a proper definition of a wave map. One may formulate the wave maps problems extrinsically by viewing $(\NN, \h)$ as an isometrically embedded submanifold of a Euclidean space $(\R^N, \ang{ \cdot, \cdot}_{\R^N})$. In this case, a \emph{wave map} is defined as a  formal critical point of the action
 \EQ{
\LL(\Phi)  =  \frac{1}{2} \int_{\R^{1+2}}  \m^{\al \be} \ang{ \p_\al \Phi , \,  \p_\be \Phi }_{\R^N} \,  \ud x\, \ud t.
}
The Euler-Lagrange equations are given by 
\EQ{
\Box \Phi \perp T_{\Phi} \NN
}
which can be rewritten as 
\EQ{ \label{ewm}
\Box \Phi = \mathcal{A}(\Phi)(\p \Phi, \p \Phi)
}
where  $\mathcal{A}$ denotes the second fundamental form of the embedding $$(\NN, \h) \hookrightarrow (\R^N, \ang{ \cdot, \cdot}_{\R^N}).$$
The conserved energy is given by 
\EQ{ \label{een} 
\E[\Phi, \p_t \Phi](t)  =  \frac{1}{2} \int_{\R^2}  \abs{\p_t \Phi(t)}^2  +  \abs{\na \Phi(t)}^2 \, \ud x  = \textrm{constant}.
}
We will study the Cauchy problem for~\eqref{ewm},  for smooth finite energy initial data $\vec \Phi(0) = (\Phi_0, \Phi_1)$ where 
\EQ{ \label{data} 
\Phi_0(x) \in \NN   \subset \R^N , \quad \Phi_1(x)  \in  T_{\Phi_0(x)}\NN, \quad  \forall  \, \, x \in \R^2.
} 
We remark that if the initial data is smooth with  finite energy, then we can find a fixed vector $\Phi_\infty \in \NN$ so that 
\EQ{\label{xlimit}
\Phi_0(x) \to \Phi_\infty \mas \abs{x} \to \infty. 
}
We also note that wave maps on $\R^{1+2}$ are called ~\emph{energy critical}, since the conserved energy and the equation are invariant under the same scaling, i.e., if $\Phi(t)$ solves~\eqref{ewm} then so does 
\EQ{
(\Phi_{\la}(t, x), \p_t \Phi_{\la}(t, x)) : = (\Phi(\la t, \la x), \, \la \p_t \Phi( \la t, \la x))
}
and we also have $\E[\vec \Phi_\la] = \E[\vec\Phi]$.

The geometry of the target manifold --  in particular the presence  of nontrivial finite energy harmonic maps $\Psi: \R^2 \to \NN$ --  plays a crucial role in the asymptotic dynamics of~\emph{large energy} solutions to the energy critical wave maps equation.  
We state the bubbling result of Sterbenz and Tataru ~\cite{ST1, ST2} below, which forms the foundation for the observations and arguments presented in this paper. 
\begin{thm}[Struwe, Sterbenz-Tataru bubbling] {\rm \cite[Theorem $1.3$ and Theorem $1.5$]{ST2}} \label{t:st} Let $\NN$ be an isometrically embedded submanifold of $\R^N$. Suppose that $\Phi(t)$ is a smooth finite energy wave map with maximal forward time of  existence $T_+$. Then either
 $\Phi(t)$ is globally regular and  scatters as $t \to \infty$, or $\Phi(t)$ bubbles  a nontrivial harmonic map at $T_+$ in the following sense: 
 \begin{itemize} 
\item If $T_+< \infty$,  there exists a sequence of times $t_n \to T_+$, a sequence of scales $\la_n = o(T_+ - t_n)$, and a sequence of translations $x_n \in \R^2$ such that 
\EQ{
\Phi_n(t, x) : =  \Phi(t_n + \la_n t, x_n+ \la_n x) 
} 
converges strongly to a Lorentz transform of a nontrivial finite energy harmonic map  $\Psi: \R^2 \to \NN$ 
 in the space $H^1_{\loc}( (-1, 1) \times \R^2; \R^N)$. 
\item  If $T_{+} = \infty$, then there exists a sequence of times $t_n \to \infty$, a sequence of scales $\la_n = o( t_n)$, and a sequence of translations $x_n \in \R^2$ such that 
\EQ{
\Phi_n(t, x) : =  \Phi(t_n + \la_n t, x_n+ \la_n x) 
} 
converges strongly to a Lorentz transform of a nontrivial finite energy harmonic map $\Psi: \R^2 \to \NN$  
in the space $H^1_{\loc}( (-1, 1) \times \R^2; \R^N)$. 
\end{itemize} 
\end{thm} 

The above result gave a resolution what was referred to as the \emph{threshold conjecture} for  $2d$ finite energy  wave maps, which we state below as a corollary.   
\begin{cor}[Threshold Theorem]{\rm \cite[Corollary $1.6$]{ST2}}\label{c:eq}
Suppose that there exists a lowest energy nontrivial harmonic map $Q: \R^2  \to \NN$. Then for smooth initial data $\vec \Phi(0) = (\Phi_0, \Phi_1)$ with 
\EQ{ \label{leq} 
\E[ \Phi_0, \Phi_1] < \E[Q, 0], 
}
the corresponding wave map evolution $\Phi(t)$ is globally regular and scatters. 
\end{cor} 

\begin{rem}\label{r:scat}
Here scattering is  meant in the sense of~\cite{ST1, ST2}. Roughly speaking there is a space-time norm $S$ defined in~\cite{ST1}, which measures dispersive properties of wave maps.  One says that a wave map $\Phi$ scatters if the $S$-norm of $\Phi$ is finite. We refer the reader to~\cite[Proposition~$3.9$]{ST1} for a characterization of wave maps with finite $S$-norm. 
\end{rem} 

\begin{rem} \label{r:T-KS}
Independent proofs of global regularity and scattering have been given by Krieger and Schlag~\cite{KS} when $\NN$ is the hyperbolic plane, and by Tao \cite{Tao37} when $\NN$ is a hyperbolic space. In both cases, note that $\NN$ is a noncompact manifold that does not admit any nontrivial harmonic maps.
 In addition, these works established uniform bounds on the scattering norm of a solution in terms of the energy and Krieger and Schlag~\cite{KS} also developed  a 
nonlinear profile decomposition for sequences of wave maps with bounded energy. 


\end{rem}
In this note we give a refinement of  Corollary~\ref{c:eq}  in the case that the target manifold  $\NN = \NN^2$ is a closed orientable surface, by taking into account  three fundamental aspects of such maps:
\begin{enumerate}
\item There is a topological invariant of a smooth finite energy wave map,  namely its \emph{topological degree}.
\item One can obtain a lower bound on the energy of a wave map in terms of its degree. 
\item The finite energy harmonic maps $\Psi: \R^2 \to \NN$ are completely classified and have minimal energy in their respective degree classes. 
\end{enumerate} 

To arrive at our refinement of Corollary~\ref{c:eq}, we first give a more precise account of $(1) - (3)$ and their implications for the sub-threshold results in this setting. 

To formulate the notion of degree, we note that given smooth finite energy initial data $(\Phi_0, \Phi_1)$, we can use~\eqref{xlimit} to identify $\Phi_0$ with a map $\ti \Phi_0: \Sp^2 \to \NN$ by assigning the point at $\infty$ to the fixed vector $\Phi_\infty:= \lim_{\abs{x} \to \infty} \Phi(x)$. 
We  then define the degree of the map $\Phi_0$ by
\EQ{
\deg(\Phi_0) := \deg( \ti \Phi_0) =  \frac{1}{ A(\NN)} \int_{\Sp^2}  \ti \Phi^*( \om) =  \frac{1}{ A(\NN)} \int_{\R^2}   \Phi^*( \om)\in \Z 
}
where $\om$ is the area element  of $\NN \subset \R^N$ and $A(\NN)$ is the area of $\NN$. 
 
We remark   that  for smooth initial data $(\Phi_0, \Phi_1)$,  $\deg(\Phi_0)$ is preserved by the smooth wave map evolution  on its maximal interval of existence $I_{\max}$, 
 i.e., 
\EQ{
\deg( \Phi_0) = \deg(\Phi(t)) \mfor t \in I_{\max}. 
} 

With this notion of degree, we  obtain a lower bound on the potential energy $\E[\Phi_0, 0]$ of initial data $(\Phi_0, \Phi_1)$  with $\deg (\Phi_0)  =  k \in \Z$,  namely  
\EQ{ \label{bf} 
\E[\Phi_0, 0]    \ge A(\NN) \abs{\deg(\Phi_0)} = \abs{k} A(\NN)
} 
The inequality~\eqref{bf} is established using the point-wise estimate~\eqref{eq:extr-vol-form-ed}. 

Next,  we recall that finite energy harmonic maps  $\R^2 \to \NN$ have been completely classified. Indeed,  Eells and Wood~\cite{EW} showed that every  harmonic map $\R^2 \to \NN$  is either holomorphic or anti-holomorphic with respect to any complex structure on $\NN$  -- here we are identifying the finite energy harmonic maps $\R^2 \to \NN$ with harmonic maps $\Sp^2 \simeq \C_{\infty}  \to \NN$, by the removable singularity theorem \cite{SU}. It then follows from standard results in complex analysis, see for example~\cite{SchlagCA}, that there exist nontrivial (i.e., non-constant) finite energy harmonic maps $\R^2 \to \NN$ if and only if $\textrm{genus}(\NN) =0$. 

If the genus of $\NN$ is zero,  the finite energy harmonic maps $\Psi:  \R^2 \to \NN$ can be identified with the rational functions $\rho: \C_\infty \to \C_{\infty}$. In particular the degree of $\Psi$ is precisely the degree of the corresponding rational function $\rho$. Moreover, for any such harmonic map $\Psi$ one has equality in~\eqref{bf} and thus 
\EQ{
\E[\Psi, 0] = A(\NN) \abs{\deg(\Psi)}
}
See Theorem~\ref{thm:holo} and Corollary~\ref{cor:high-genus} and Corollary~\ref{cor:zero-genus} below for precise statements of the preceding results on harmonic maps. 
 
Now we are nearly ready to state the main result of this note. From our observations above we see that if $\NN$ is a closed oriented surface,  Corollary~\ref{c:eq} is only meaningful if the genus of $\NN$ is zero, since otherwise there are no nontrivial harmonic maps. In the genus zero case, the minimal energy nontrivial harmonic maps can be identified with the degree one rational functions, $\rho : \C_{\infty} \to \C_{\infty}$, i.e., the M\"obius transformations. From now on we will denote such a harmonic map by $Q$, and we will call any such $Q$ a \emph{ground state harmonic map}. We have 
\EQ{ \label{gs} 
\E[Q, 0] = A(\NN)
}
It follows from~\eqref{bf} and~\eqref{gs} that the hypothesis~\eqref{leq} in the Threshold Theorem~\ref{c:eq}~\emph{only applies to initial data} $(\Phi_0, \Phi_1)$ \emph{with} $\deg(\Phi_0) = 0$.   Our main observation is that in this case one can do better than~\eqref{leq}.

\begin{thm}[Refined Threshold Theorem] \label{t:main} 
Let $(\NN, \h)$ be an orientable closed surface with $\textrm{genus}(\NN) = 0$.  Suppose that $\vec \Phi(0)  = (\Phi_0, \Phi_1)$ is smooth finite energy initial data with 
\EQ{ \label{<2EQ}
\deg(\Phi_0) = 0 \mand \E[\Phi_0, \Phi_1] < 2 \E[Q, 0]
}
where $Q: \R^2 \to \NN$ is the ground state harmonic map. Then the wave map evolution  $\Phi(t)$ of the data $\vec \Phi(0)$  is globally regular, and scatters as $ t \to \pm \infty$. 
\end{thm} 

The heuristic reason behind the threshold $2 \E[Q, 0]$ is as follows. The degree counts the number of times a map `wraps around' the target with orientation taken into account. Suppose that a harmonic map of degree $k$ bubbles off from a wave map $\Phi$. Then in order for $\Phi$ to have degree zero, it must `unwrap' precisely $\abs{k}$ times away from the bubble. The minimum energy cost for doing so is $\abs{k} A(\NN)$, which is also the energy of the degree $k$ harmonic map as discussed above. The case of $k = 1$ is the first nontrivial harmonic map (i.e., the ground state $Q$), and thus we are led to the threshold $2 \E[Q, 0]$.

\begin{rem} \label{r:equiv}
We remark that the analog of Theorem~\ref{t:main} for equivariant wave maps was proved without making use of~Theorem~\ref{t:st}  in~\cite[Theorem~$1.1$]{CKLS1}  via a version of the concentration compactness/ridigity method of Kenig and Merle~\cite{KM06, KM08} together  with Struwe's classic bubbling theorem~\cite{Struwe}. The point here is that the main result in~\cite{Struwe} is not as strong as the restriction of Theorem~\ref{t:st} to the equivairant setting since it comes without a scattering statement.  

We also note that Theorem~\ref{t:main} is sharp in the case of finite time blow-up. Indeed the explicit blow-up constructions of Krieger, Schlag, Tataru~\cite{KST}, Rodnianski, Sterbenz~\cite{RS}, and Rapha\"el, Rodnianski~\cite{RR} can be easily modified to produce degree zero blow-up solutions with energy slighty above $2\E[Q, 0]$; see~\cite[Section~$3.1$]{CKLS1}. 
\end{rem}

\subsection{Yang-Mills} 
It is natural to ask if one can formulate a version of this refined threshold conjecture in the setting of the Yang-Mills equation for $SU(2)$ vector bundles over $\R^{1+4}$, where the \emph{second Chern number} plays the role of the topological degree.  Of course, a result in the spirit of Theorem~\ref{t:st} is not yet known for Yang-Mills. However, even assuming that an analog of Theorem~\ref{t:st} does hold, there are additional difficulties one must take into account. For example, the  finite energy static solutions are not classified as in the case of harmonic maps between closed orientable surfaces. Moreover, it was shown in~\cite{SSU} that there are finite  energy stationary solutions -- even with vanishing second Chern number --  to the Yang-Mills equation which do not minimize the energy. For a discussion of how to properly formulate the threshold conjecture for Yang-Mills we refer the reader to Section~\ref{sec:ym}. 

 
\section{Proof of refined threshold results for wave maps} \label{sec:wm}
We identify $\bbR^{2}$ with $\bbC$ and $\bbS^{2}$ with the Riemann sphere $\bbC_{\infty}$. By the removable singularity theorem~\cite{SU}, a finite energy harmonic map from $\bbR^{2}$ extends to a harmonic map from $\bbC_{\infty}$. The following classification theorem is known for harmonic maps from $\bbC_{\infty}$.

\begin{thm} \label{thm:holo}
Every finite energy harmonic map from $\bbC^{\infty}$ to $\calN$ is holomorphic or anti-holomorphic (with respect to any complex structure on $\calN$).
\end{thm}

\begin{proof} 
By H\'elein's regularity theorem \cite{Hel} such harmonic maps are smooth. Then the theorem follows from a classical argument of Eells-Wood \cite{EW}.
\end{proof}

\begin{cor} \label{cor:high-genus}
If the genus of $\calN$ is greater than $0$, then there does not exist a  non-constant finite energy harmonic map from $\bbS^{2}$ to $\calN$.
\end{cor}

\begin{proof} 
The notion of a holomorphic (or an anti-holomorphic) map depends only on the conformal structure of both the domain and the target. Hence this statement is an easy consequence of the Riemann-Hurwitz formula from complex analysis, from which one deduces that there are no non-constant 
holomorphic maps from $\bbC_{\infty}$ to a genus $g > 0$ Riemann surface. See for example~\cite{SchlagCA}.  \qedhere 
\end{proof}

\begin{cor} \label{cor:zero-genus}
Consider a smooth orientable closed surface $\calN$ of genus $0$. Then the following statements hold.
\begin{enumerate}
\item  \label{part:hm} The finite energy harmonic maps from $\bbC_{\infty}$ to $\calN \simeq \C_{\infty}$ can be identified with  the rational functions $\bbC_{\infty} \to \bbC_{\infty}$, where we use the complex coordinate $w$ given by the stereographic projection on $\C_\infty$.
\item \label{part:zero-deg} Every degree $0$ harmonic map is constant. 
\item \label{part:high-deg} The energy of a degree $k$ harmonic map equals $\abs{k} A(\calN)$. In particular, $\calE[Q, 0] = A(\calN)$.
\end{enumerate}
\end{cor}
\begin{proof}
Statements~{\it(\ref{part:hm})} and~{\it(\ref{part:zero-deg})} are standard results in complex analysis and ${\it(3)}$ is an easy consequence of ${\it(1)}$ and $\it{(2)}$. 
\end{proof} 
As in Theorem~\ref{t:st}, let $(\calN, \bfh)$ be given as an isometrically embedded surface in $\bbR^{N}$, which we still denote by $\calN$. The following is the key lemma for our proof of Theorem~\ref{t:main}. 
\begin{lem} \label{lem:key}
Let $I=(-1, 1)$. Let $\Phi_{(n)}$ be a sequence of smooth maps $I \times \bbR^{2} \to \calN \subseteq \bbR^{N}$ such that $\deg \, \Phi_{n} \rst_{t} = 0$ for every $t \in I$, and
let $\Psi$ be a smooth map $I \times \bbR^{2} \to \calN \subseteq \bbR^{N}$ such that $\mathrm{deg} \, \Psi \rst_{t} = k \neq 0$ for every $t \in I$. Suppose that we have the convergence
\begin{equation}
	\Phi_{(n)} \to \Psi \quad \hbox{ in } H^{1}_{\mathrm{loc}}(I \times \bbR^{2}; \bbR^{N})
\end{equation}
Then the following lower bound on the time-average of the energy of $\Phi_{(n)}$ holds:
\begin{equation}
	\limsup_{n \to \infty} \int_{-\frac{1}{2}}^{\frac{1}{2}} \calE[\Phi_{(n)}(t), 0] \, \ud t 
	\geq 2 \abs{k} A(\calN).
\end{equation}
\end{lem}

\begin{proof} 
We begin with a few general formulae in the extrinsic setting. Given a smooth map $\Phi: \bbR^{2} \to \calN \subseteq \bbR^{N}$, the pullback of the volume form $\omg$ of $\calN$ takes the form
\begin{equation} \label{eq:extr-vol-form}
	\Phi^{\ast}\omg(x) = \omg \vert_{\Phi(x)} (\rd_{1} \Phi \wedge \rd_{2} \Phi)(x) \, \ud x^{1} \wedge \ud x^{2},
\end{equation}
where $\omg \vert_{\Phi(x)}$ is the volume form of $\calN$ at $\Phi(x)$. By the embedding $\calN \subseteq \bbR^{N}$, we may view $\omg$ as a $\wedge^{2} \bbR^{N}$-valued function on $\calN$ such that $\abs{\omg}_{\wedge^{2}\bbR^{N}} = 1$. The degree of the map $\Phi$ is given by the integral
\begin{equation} \label{eq:extr-deg}
	\deg (\Phi) = \frac{1}{A(\calN)} \int_{\bbR^{2}} \omg \vert_{\Phi(x)} (\rd_{1} \Phi \wedge \rd_{2} \Phi)(x) \, \ud x^{1} \wedge \ud x^{2}.
\end{equation}
Finally, we observe that we have the pointwise bound
\begin{equation} \label{eq:extr-vol-form-ed}
	\abs{\omg \vert_{\Phi(x)} (\rd_{1} \Phi \wedge \rd_{2} \Phi)(x)} \leq \frac{1}{2} \bb( \brk{\rd_{1} \Phi, \rd_{1} \Phi}(x) + \brk{\rd_{2} \Phi, \rd_{2} \Phi} (x) \bb),
\end{equation}
which is an immediate consequence of $\abs{\omg}_{\wedge^{2}\bbR^{N}} = 1$. 

We now begin the proof in earnest. Assume that $\deg (\Psi) > 0$; the case $\deg (\Psi) < 0$ can be dealt with similarly. Fix any $\eps > 0$, $R > 0$ and let $I_{0} = [-1/2, 1/2]$. For simplicity, we shall abuse the notation slightly and write $\Phi^{\ast}\omg(x) = \omg \vert_{\Phi(x)} (\rd_{1} \Phi \wedge \rd_{2} \Phi)(x)$ in view of \eqref{eq:extr-vol-form}. We claim that up to passing to a subsequence we have
\begin{equation} \label{eq:key-claim}
	\int_{I_{0}} \int_{B_{R}} \abs{\Phi_{(n)}^{\ast} \, \omg - \Psi^{\ast} \, \omg} \, \ud x^{1} \ud x^{2} \, \ud t < \frac{\eps}{4} \quad \hbox{ for sufficiently large } n.
\end{equation}
Indeed, by hypothesis, we have the strong convergence
\begin{equation} \label{eq:strongH1}
 	\Phi_{(n)} \to \Psi \quad \hbox{ in } H^{1}(I_{0} \times B_{R}; \bbR^{N})
\end{equation}
as $n \to \infty$. Using the formula \eqref{eq:extr-vol-form}, we split the integrand in \eqref{eq:key-claim} into two as follows:
\begin{align}
\Phi_{(n)}^{\ast} \, \omg - \Psi^{\ast} \, \omg
=&	\omg \vert_{\Phi_{(n)}} \bb( \rd_{1} \Phi_{(n)} \wedge \rd_{2} \Phi_{(n)} - \rd_{1} \Psi \wedge \rd_{2} \Psi \bb)		\label{eq:key-term1} \\
& + (\omg \vert_{\Phi_{(n)}} - \omg \vert_{\Psi}) (\rd_{1} \Psi \wedge \rd_{2} \Psi). \label{eq:key-term2}
\end{align}
The contribution of \eqref{eq:key-term1} to the left-hand side of \eqref{eq:key-claim} goes to $0$ as $n \to \infty$ by \eqref{eq:strongH1} and the fact that $\abs{\omg}_{\wedge^{2} \bbR^{N}} = 1$.
For the contribution of \eqref{eq:key-term2}, we first observe the following: By the $L^{2}(I_{0} \times B_{R})$ convergence $\Phi_{(n)} \to \Psi$, there exists a subsequence (still denoted by $\Phi_{(n)}$) such that
\begin{equation*}
	\Phi_{(n)} \to \Psi \quad \hbox{ almost everywhere on } I_{0} \times B_{R}.
\end{equation*}
Therefore $\omg \vert_{\Phi_{(n)}} \to \omg \vert_{\Psi}$ almost everywhere on $I_{0} \times B_{R}$ as well.
Since $\rd_{1} \Psi \wedge \rd_{2} \Psi$ is a fixed integrable function on $I_{0} \times B_{R}$, it follows from the dominated convergence theorem that 
\begin{equation*}
	\int_{I_{0}} \int_{B_{R}} \abs{(\omg \vert_{\Phi_{(n)}} - \omg \vert_{\Psi}) (\rd_{1} \Psi \wedge \rd_{2} \Psi)} \, \ud x^{1} \ud x^{2} \ud t \to 0 \quad \hbox{ as } n \to \infty,
\end{equation*}
which proves the claim.

From \eqref{eq:key-claim}, up to passing to a subsequence and for sufficiently large $n$, we obtain the one-sided bound
\begin{equation*}
	\int_{I_{0}} \bb( \int_{B_{R}} \Phi_{(n)}^{\ast} \, \omg\bb) \, \ud t \geq \int_{I_{0}} \bb( \int_{B_{R}} \Psi^{\ast} \, \omg \bb) \, \ud t - \frac{\eps}{4}.
\end{equation*}
Since the degree of the map $\Phi_{(n)}$ restricted to every constant $t$-hypersurface is always zero, we also have
\begin{equation*}
	- \int_{I_{0}} \bb( \int_{\bbR^{2} \setminus B_{R}} \Phi_{(n)}^{\ast} \, \omg\bb) \, \ud t 
	= \int_{I_{0}} \bb( \int_{B_{R}} \Phi_{(n)}^{\ast} \, \omg\bb) \, \ud t \geq \int_{I_{0}} \bb( \int_{B_{R}} \Psi^{\ast} \, \omg \bb) \, \ud t - \frac{\eps}{4}.
\end{equation*}
Using the pointwise bound \eqref{eq:extr-vol-form-ed}, we then arrive at the lower bound
\begin{align*}
	\int_{I_{0}} \calE[\Phi_{(n)}, 0] \, \ud t
	= &\int_{I_{0}} \int_{\bbR^{2}} \frac{1}{2} \bb( \brk{\rd_{1} \Phi, \rd_{1} \Phi} + \brk{\rd_{2} \Phi, \rd_{2} \Phi}  \bb) \, \ud x^{1} \ud x^{2} \, \ud t \\
	\geq & \int_{I_{0}} \bb( \int_{B_{R}} \abs{\Phi^{\ast}_{(n)} \omg} \, \ud x^{1} \ud x^{2} 
						+ \int_{\bbR^{2} \setminus B_{R}} \abs{\Phi^{\ast}_{(n)} \omg} \, \ud x^{1} \ud x^{2} \bb) \, \ud t \\
	\geq & 2 \int_{I_{0}} \bb( \int_{B_{R}} \Psi^{\ast} \, \omg \bb) \, \ud t
			 - \frac{\eps}{2}.
\end{align*}
Choosing $R$ sufficiently large, we may ensure that
\begin{equation*}
\int_{B_{R}} \Psi^{\ast} \, \omg 
\geq \int_{\bbR^{2}} \Psi^{\ast} \, \omg  - \frac{\eps}{4}
= A(\calN) \deg(\Psi) - \frac{\eps}{4}.
\end{equation*}
for every $t \in I_{0}$. As $\eps > 0$ is arbitrary, the conclusion of the lemma follows. \qedhere
\end{proof}

With Corollaries~\ref{cor:high-genus}, \ref{cor:zero-genus} and Lemma~\ref{lem:key} in hand, we are ready to prove Theorem~\ref{t:main}.
\begin{proof} [Proof of Theorem~\ref{t:main}]
Suppose that Theorem~\ref{t:main} fails. Then by Theorem~\ref{t:st}, we obtain a sequence $\Phi_{(n)} : (-1, 1) \times \bbR^{2} \to \calN$ of degree zero smooth wave maps converging in $H^{1}_{\mathrm{loc}}$ to $\Psi : (-1, 1) \times \bbR^{2} \to \calN$, which is a Lorentz transform of a time-independent harmonic map $\Psi_0: \bbR^{2} \to \calN$. 
By Corollary~\ref{cor:zero-genus}{\it (\ref{part:zero-deg})} we may assume that the degree of $\Psi_0 $ is $k \neq 0$, and it follows  that $\deg(\Psi)\rst_{t}  = k$ for any time $t \in I$. Then applying Lemma~\ref{lem:key} and the fact that the energy of $\Phi_{(n)}$ is independent of $n \in \N$ and $t \in I$,  we deduce that 
\EQ{
\limsup_{n \to \infty}  \E[\Phi_{(n)}, 0] \ge 2 \abs{k} A(\NN) = 2 \abs{k} \E[Q, 0]
} 
where we have used Corollary~\ref{cor:zero-genus}{\it (\ref{part:high-deg})} to relate $A(\calN)$ with $\calE[Q, 0]$. This contradicts~\eqref{<2EQ} and thus the theorem is proved. \qedhere
\end{proof}

\section{Threshold conjecture for Yang-Mills} \label{sec:ym}

Consider the Lie group $\mathrm{SU}(2)$ of $2 \times 2$ special unitary matrices, with the associated Lie algebra $\mathrm{su}(2)$ consisting of $2 \times 2$ trace-free anti-hermitian matrices. Let $\eta$ be a $\mathrm{SU}(2)$ vector bundle over $\bbR^{1+4}$. The Yang-Mills equations for a connection $A$ on $\eta$ is the Euler-Lagrange equation for the action
\begin{equation} \label{eq:YM-S}
	\calL_{\mathrm{YM}}(A) := \frac{1}{8} \int_{\bbR^{1+4}} \brk{F_{\mu \nu}[A], F^{\mu \nu}[A]} \, \ud t \ud x.
\end{equation}
where $F[A]$ is the $\mathrm{su}(2)$-valued curvature 2-form associated to $A$ and $\brk{F, G} := \mathrm{tr} \, (F G^{\dagger})$. 
The conserved energy takes the form
\begin{equation} \label{eq:YM-E}
	\calE[A](t) := \frac{1}{8} \int_{\bbR^{4}} \sum_{\mu, \nu \in \set{0, 1,2,3,4}} \brk{F_{\mu \nu}[A(t)], F_{\mu \nu}[A(t)]} \, \ud x.
\end{equation}
Restricting to connections which are independent of $t$, the Yang-Mills action coincides with the energy, and the resulting Euler-Lagrange equation give rise to an elliptic PDE. This system has played an important role in the study of smooth $4$-manifolds via the Donaldson theory \cite{DonaldsonKronheimer}, and in the literature it is also often referred to as the Yang-Mills equation. To distinguish it from the time-dependent Yang-Mills equation on $\bbR^{1+4}$, we shall refer to this elliptic PDE as the \emph{elliptic Yang-Mills equation}, and to its solutions as \emph{elliptic Yang-Mills connections}.

The Yang-Mills connections on $\bbR^{1+4}$ exhibit many similarities with the wave maps from $\bbR^{1+2}$.  Similar to the concept of degree of a map, associated to each $\mathrm{su}(2)$ vector bundle $\eta$ over $\bbR^{4}$ is its \emph{second Chern number} (often also called the \emph{topological charge}) defined as
\begin{equation} \label{eq:YM-N}
	N := c_{2} (\eta)[\bbR^{4}] = \frac{1}{8 \pi^{2}} \int_{\bbR^{4}} \tr (F[A] \wedge F[A]),
\end{equation}
where $A$ is any connection on $\eta$. In what follows, we shall also refer to $N$ as the second Chern number of a connection $A$. As the notation suggests, the number $N$ is given by integrating over $\bbR^{4}$ a $4$-form
\begin{equation} \label{eq:c2}
	c_{2} (\eta) = \frac{1}{8 \pi^{2}} \tr (F[A] \wedge F[A]),
\end{equation}
which is called the \emph{second Chern class} of $\eta$. 

Minimizers of the (time-independent) Yang-Mills energy on $\bbR^{4}$ in the class of connections with a given second Chern number $N$ give rise to special elliptic Yang-Mills connections, namely \emph{self-dual} or \emph{anti-self-dual} connections. In particular, the anti-self-dual connections on a $\mathrm{SU}(2)$ bundle with $N = 1$ are often referred to as the \emph{first instantons}. The energy of a self- or anti-self-dual connection equals $\pi^{2} \abs{N}$. These connections may be considered as analogs of holomorphic/anti-holomorphic maps in dimension two. We refer to \cite[Chapter 10]{MS} for more detailed introduction to these concepts. 
 
Although the analog of Theorem~\ref{t:st} has not yet been established for the Yang-Mills equation, it is natural to ask whether a refined threshold statement like Theorem~\ref{t:main} also holds for the Yang-Mills equation. For the purpose of discussion, we shall assume that the following conjecture holds:
\begin{conj}[Bubbling for Yang-Mills] \label{conj:YM}
Suppose that $A(t)$ is a smooth finite energy solution to the Yang-Mills equation with maximal forward time of existence $T_{+}$. Then either $T_{+} = \infty$ or $A(t)$ bubbles a nontrivial elliptic Yang-Mills connection at $T_{+}$ in the following sense:
\begin{itemize}
\item There exists a sequence of times $t_{n} \to T_{+}$, a sequence of scales $\lmb_{n} = o(T_{+} - t_{n})$, and a sequence of translations $x_{n} \in \bbR^{4}$ such that
\begin{equation*}
	A_{(n)}(t,x) := \lmb_{n} A(t_{n} + \lmb_{n} t, x_{n} + \lmb_{n} x)
\end{equation*}
converges to a Lorentz transform of a nontrivial finite energy elliptic Yang-Mills connection $B$ in the following sense:
\begin{equation*}
	F[A_{(n)}] \to F[B] \hbox{ as } n \to \infty \hbox{ in } L^{2}_{\mathrm{loc}}((-1, 1) \times \bbR^{4}).
\end{equation*}
\end{itemize}
\end{conj}

\begin{rem}
When $T_{+} = \infty$, it is expected that an analog of the second part of Theorem~\ref{t:st} (i.e., the dichotomy of bubbling and scattering) holds as well. We however restrict our attention to the question of global regularity (i.e., whether $T_{+} = \infty$ or not) for simplicity.
\end{rem}

Our proof of Lemma~\ref{lem:key} in Section~\ref{sec:wm} can be easily adapted to the case of the Yang-Mills equation, where the concept of degree is replaced by the second Chern number $N$ and the pullback volume form $\Phi^{\ast} \omg$ by the second Chern class $c_{2}(\eta)$. Combined with Conjecture~\ref{conj:YM}, we obtain the following proposition.

\begin{prop} \label{prop:YM-key}
Suppose that Conjecture~\ref{conj:YM} is true. Let $A(t)$ be a smooth finite energy Yang-Mills connection with vanishing second Chern number. Suppose that the maximal forward time of existence of $A$ is finite, i.e., $T_{+} < \infty$, and let $B$ be the Lorentz transformed elliptic Yang-Mills connection given by Conjecture~\ref{conj:YM}. If the second Chern number $N$ of $B$ is nonzero, then we have the lower bound
\begin{equation*}
	\calE[A] \geq 2 \pi^{2} \abs{N}.
\end{equation*}
\end{prop}

In the case of wave maps, the full refined threshold theorem (Theorem~\ref{t:main}) then follows from the observation that there does not exist any nontrivial harmonic maps with degree $0$ (see Corollary~\ref{cor:zero-genus}). However, the case of Yang-Mills is different: It is known that there exist nontrivial solutions to the time-independent Yang-Mills equations with zero second Chern number which are not self- or anti-self-dual. We refer to \cite{SSU}; see also \cite[Section~11.6]{MS} and references therein for further discussion. 

In conclusion, we have shown that if Conjecture~\ref{conj:YM} holds, then the bubbling of an elliptic Yang-Mills connection with nonzero second Chern number requires at least twice the energy of the first instanton, in analogy to the case of wave maps. The optimal energy threshold for global regularity of Yang-Mills connections with $N = 0$ therefore seems to hinge on proving a lower bound for the energy of non-minimal solutions to the elliptic Yang-Mills equations with vanishing second Chern number.\footnote{There is a question of whether it is possible to bubble a non-minimal elliptic Yang-Mills connection dynamically from a smooth data. A priori there is no reason to preclude this scenario. In fact, in the model case of the energy critical NLW $\Box u = u^5$ on $\bbR^{1+3}$, an explicit construction of a blow-up by a non-minimal static solution is known \cite{KST3}.} 

\bibliographystyle{plain}
\bibliography{researchbib}

\bigskip

\centerline{\scshape Andrew Lawrie, Sung-Jin Oh}
\smallskip
{\footnotesize
 \centerline{Department of Mathematics, The University of California, Berkeley}
\centerline{970 Evans Hall \#3840, Berkeley, CA 94720, U.S.A.}
\centerline{\email{ alawrie@math.berkeley.edu, sjoh@math.berkeley.edu}}
} 

\end{document}